\documentclass[oneside,english]{elsarticle}
\makeatletter
\def\ps@pprintTitle{%
 \let\@oddhead\@empty
 \let\@evenhead\@empty
 \def\@oddfoot{\centerline{\thepage}}%
 \let\@evenfoot\@oddfoot}
\makeatother
\usepackage[T1]{fontenc}
\usepackage[utf8]{inputenc}
\usepackage{amsfonts}
\usepackage{amssymb}
\usepackage{amsmath}
\usepackage{amsthm}
\usepackage[mathscr]{euscript}
\usepackage{enumerate}
\usepackage[left=3cm, right=3cm, top=3cm, bottom=3cm]{geometry}

\newtheorem{theorem}{Theorem}

\newtheorem{definition}[theorem]{Definition}
\newtheorem{example}[theorem]{Example}

\newtheorem{lemma}[theorem]{Lemma}

\newtheorem{proposition}[theorem]{Proposition}
\newtheorem{remark}[theorem]{Remark}

\let\ds\displaystyle

\def\n{N}
\def\t{T}
\def\w{W}
\def\v{V}

\def\F{F}
\def\G{G}
\def\H{H}

\def\N{{\bar N}}
\def\T{{\bar T}}
\def\W{{\bar W}}
\def\V{{\bar V}}

\def\k{\kappa}
\def\p{\mathfrak{p}}
\def\f{\mathfrak{f}}
\def\g{\mathfrak{g}}
\def\a{\mathfrak{a}}

\begin{document}
\title{A Lie algebra structure on variation vector fields along curves in $2$-dimensional space forms}
\author[label1]{Jos\'e del Amor}
\author[label2]{\'Angel Gim\'enez}
\author[label1]{Pascual Lucas}
\address[label1]{Departamento de Matem\'aticas, Universidad de Murcia. \\ Campus del Espinardo, 30100 Murcia, Spain}
\address[label2]{Centro de Investigaci\'on Operativa, Universidad Miguel Hern\'andez de Elche.  \\ Avda. Universidad s/n, 03202 Elche (Alicante), Spain}

\date{}

\begin{abstract}
	A Lie algebra structure on variation vector fields along an immersed curve in a $2$-dimensional real space form is investigated. This Lie algebra particularized to plane curves is the cornerstone in order to define a Hamiltonian structure for plane curve motions. The Hamiltonian form and the integrability of the planar filament equation is finally discussed from this point of view.
\end{abstract}


\maketitle

\section{Introduction}
This paper is motivated by the investigation of relationships between geometric motions of curves in a certain space and Hamiltonian systems of PDE's. Hamiltonian systems have been and remain an active topic of research, as witnessed by a number of relevant publications in the last decades. There is a great number of papers dealing with the study of  integrable systems associated with the evolution of plane (or $n$-dimensional) curves, specifically, evolution equations for the differential invariants of the curve, for example, its curvatures. The pioneering work \cite{hasimoto_soliton_1972} by Hasimoto about the motion of a vortex filament in a fluid and its connection with the cubic nonlinear Schrödinger (NLS) equation through the Hasimoto transformation  was the first to suggest such a link. The \emph{vortex filament flow} is the evolution equation for space curves modelling the motion of a one-dimensional vortex filament in an incompressible fluid, and is given by
\begin{equation}\label{LIE}
	\gamma_t=kB,
\end{equation} 
where $\gamma_t(s)=\gamma(s,t)$ denotes the evolving curve, parametrized by the arc-length $s$, and $k$ is its curvature. 

Motivated by this result, many other authors raised the issue of what other curves flows induce a  completely integrable PDE for their curvatures. In \cite{lamb_solitons_1977} Lamb gave a general procedure that helps to identify a certain space curve evolution with a given integrable equation. Later, in \cite{gurses_motion_1998,beffa_theory_1999,beffa_poisson_2010,beffa_integrable_2002,sanders_integrable_2003,anco_bi-hamiltonian_2006} and other references therein, the authors provide Hamiltonian structures to the intrinsic geometry of curves in $n$-dimensional Riemannian manifolds from different backgrounds. Chou and Qu (see \cite{chou_integrable_2002-1,chou_integrable_2003}) study motions of plane curves when the velocity vector fields are given in terms of differential invariants under other background geometries, considering Klein geometries instead of Euclidean ones. They have obtained, among other nonlinear equations, that KdV, Harry-Dym and Sawada-Kotera equations are induced by this kind of motions. Nevertheless, the common thread of these papers is the study of the relationship between differential invariants and Hamiltonian structures of PDE's.
 
However, there are very few papers addressing the integrability of the curve motion equation itself. In this regard, Langer and Perline (see for instance \cite{langer_poisson_1991}) interpreted the Hasimoto transformation as a Poisson map transforming LIE to NLS, which allows to lift Poisson structures of NLS equation to appropriate Poisson structures on the space of curves, finding in this way an infinite hierarchy of local commuting generalized symmetries and conserved quantities in involution for the LIE equation as the pull-back of those for the NLS equation. Yasui and Sasaki (see \cite{yasui_differential_1998}), based on \cite{langer_poisson_1991}, introduce a differential calculus on the space of asymptotically linear curves in order to clarify the integrability of the vortex filament equation by developing a recursion operator, symmetries and constants of motion. The \emph{planar filament (PF)} equation (\cite{langer_planar_1996,nakayama_integrability_1992}) is given by
\begin{equation}
\gamma_t=\frac12 k^2 {\t}+k' {\n}.	
\end{equation}
In \cite{langer_planar_1996}, the authors show that all the odd vector fields in the infinite sequence of commuting vector fields for LIE hierarchy are planarity-preserving, and this infinite sequence, starting with PF itself, is an integrable system. The PF flow induces the mKdV equation $k_t=k'''+\frac{3}{2}k^2k'$, a well-known completely integrable PDE. 

Drawing inspiration from the ideas of these papers, our aim here is to define a Lie algebra structure on the whole set of local variation vector fields along a inmersed curve in a $2$-dimensional space form. It should be remarked that the subspace consisting of local arc-length preserving variation vector fields is actually closed under bracket (Theorem \ref{Lie-Bracket}). By using this Lie algebra particularized to plane curves we define a Lie algebra structure on the phase space of PF equation, i.e., on the arc-length preserving variation vector fiels on planar curves (Theorem \ref{Lie-Bracket2}), that enables us to set a Hamiltonian structure as well as showing its integrability. To achieve this, we first construct a Lie algebras homomorphism between the Lie algebra of local arc-length variation vector fields and the Lie algebra of derivation vector fields. Then, we pull-back the standard Hamiltonian structure at the level of the curvature flow (the one for mKdV flow) to a Hamiltonian structure at the level of the curve flow (the one for PF equation) by means of such homomorphism. Our approach may be somewhat a formalization of the idea given in \cite{ivey_integrable_2001} in which a curve motion flow is defined to be integrable if it induces a completely integrable system of PDE for its curvatures. Our background will be the algebraic general formalism by Gelfand and Dickey (see \cite{dickey_soliton_2003} for more details) about Hamiltonian structures, even though we could have used another different algebraic framework for our purposes. 
%One can define integrability of a curve evolution by saying that its associated curve evolution is integrable, but this misses the Poisson geometry of the problem. 
%In particular, we give an explicit Hamiltonian structure on the phase space consisting of variation vector fields on plane curves. 

Lastly, let us add that it was shown in \cite{goldstein_solitons_1992} that the PF equation can be interpreted physically as localized induction equation for boundary of vortex patch for ideal fluid flow in two dimensions. But, the PF equation has appeared in many other contexts, for instance, in \cite{maksimovic_multisolitons_2003} the authors present the results of experimental and theoretical study of two-dimensional vortex filament tangling, appearing in laser-matter interactions on nanosecond time scale. In two dimensions, the authors use planar filament equation as an evolution equation on plane curves. 

%If we compare the curve shortening flow $\gamma_t=k\n$ with the PF equation, the difference is that while PF equation is integrable and possesses infinitely many conservation laws (\cite{langer_planar_1996}), the former is a dissipative parabolic equation. The PF equation is related to the nonlinear evolution equation solved by the inverse scattering method (\cite{nakayama_integrability_1992}). 

Below we summarize the general outline of the rest of the paper. In Sect. 2 we introduce some definitions and properties about derivations, and afterwards we give the suitable mathematical framework to study our Hamiltonian structures. Then we move in Sect. 3 to study the space of immersed arc-length parametrized curves in a $2$-dimensional real space form. In particular, we describe the space consisting of variation vector fields locally preserving arc-length parameter and derive important results for later use. Sect. 4 is the main section and is devoted to study the Lie algebra structure on local variation vector fields. In Sect. 5 we use such a Lie algebra to construct the Hamiltonian operator and give the Hamiltoninan structure for planar filament equation, discussing also its integrability.

\section{Preliminaries}
In this section we present the algebraic general framework underlying in both finite and infinite dimensional Hamiltonian systems (see \cite{dickey_soliton_2003}). Firstly, we introduce the differential calculus which will be needed later.
\subsection{Differential algebra $\mathcal{P}$}
Let $n$ be a positive integer and consider $u_0,u_1,\ldots,u_{n-2}$ differentiable functions in the real variable $x$. Set
\begin{equation*}
	u_i^{(m)}=\frac{d^mu_i}{dx^m},\quad \text{for } m\in \mathbb{N}, i\in \left\{0,\ldots,n-2\right\}.
\end{equation*}

Let $\mathcal{P}$ be the real algebra of polynomials in $u_0,\ldots,u_{n-2}$ and their derivatives of arbitrary order, namely,
\begin{equation*}
	\mathcal{P}=\mathbb{R}[u^{(m)}_i:m\in \mathbb{N},i\in \left\{0,\ldots,n-2\right\}].
\end{equation*}

In this algebra we define a derivation $\partial$ obeying
\begin{equation*}
	\begin{cases} 
		\partial(f g)=(\partial f)g+f(\partial g), \\
		\partial(u_i^{(m)})=u_i^{(m+1)},
	\end{cases}
\end{equation*}
becoming $\mathcal{P}$ a differential algebra. We also denote by $\mathcal{P}_0$ the elements of $\mathcal{P}$ such that its constant term vanishes. It is customary to take $\partial$ as the total derivative $D_x$ which can be viewed as
\begin{equation*}
	D_x=\sum_{i=0}^{n-2}\sum_{m\in \mathbb{N}}u_i^{(m+1)} \frac{\partial }{\partial u_i^{(m)}}.
\end{equation*}

Besides $\partial$, other derivations $\xi$ may also be considered. The action of $\xi$ is determined if we know how $\xi$  acts on the generators of the algebra. Indeed, set 
$$a_{i,m}=\xi u_i^{(m)},$$
then, for any $f\in \mathcal{P}$ we have
$$\xi f=\sum_{i=0}^{n-2}\sum_{m\in\mathbb{N}} a_{i,m} \frac{\partial f}{\partial u_i^{(m)}}.$$
Derivations commuting with the total derivative have important properties. Among others, if $[\xi,\partial]=0$, we have
\begin{equation*}
	a_{i,m+1}=\xi u_i^{(m+1)}=\xi\partial u_i^{(m)}=\partial\xi u_i^{(m)}=\partial a_{i,m}.
\end{equation*}
Thus 
\begin{equation*}
	\xi f=\sum_{i=0}^{n-2}\sum_{m\in\mathbb{N}} a_i^{(m)} \frac{\partial f}{\partial u_i^{(m)}},
\end{equation*}
where $a_i=a_{i,0}=\xi u_i$.
The space of all derivations on $\mathcal{P}$, denoted by $\mathop{\rm der}\nolimits(\mathcal{P})$, is a Lie algebra with respect to the usual commutator
\begin{equation*}
	[\partial_1,\partial_2]=\partial_1\partial_2-\partial_2\partial_1,\quad \partial_1,\partial_2\in \mathop{\rm der}\nolimits(\mathcal{P}).
\end{equation*}

Let $a=(a_0,\ldots,a_{n-2})$ be a set of $n-1$ elements of $\mathcal{P}$ and write
\begin{equation}\label{derivation}
	\partial_a=\sum_{i=0}^{n-2}\sum_{m\in\mathbb{N}} a_i^{(m)}\frac{\partial }{\partial u_i^{(m)}}.
\end{equation}
A direct computation shows that $\partial_a$ are derivations in $\mathcal{P}$ verifying
\begin{equation}\label{LieBracketDerivation}
	[\partial_a,\partial_b]=\partial_{\partial_ab-\partial_ba}.
\end{equation}
Thus, the set of derivations $\partial_a$ is a Lie subalgebra of $\mathop{\rm der}\nolimits(\mathcal{P})$. We will refer to $\partial_{a}$ as a derivation vector field, and the algebra of derivation vector fields will be denoted by  $\mathop{\rm der}\nolimits^*(\mathcal{P})$. Observe that, in particular, if we take $a=u'=(u'_0,\ldots,u'_{n-2})$ then $\partial=\partial_{u'}$, being $u'=D_x u$.

%Finally we shall define the integration over $\mathcal{P}$ at a formal level. We define the integral $\int f dx$ by requiring two properties: linearity and $\int f' dx=0$. Hence the integral can be seen as an homomorphism of a linear space $\mathcal{P}$ onto another linear space $\mathcal{L}$ whose kernel is $\partial \mathcal{P}$. The quotient space $\tilde{ \mathcal{P}}=\mathcal{P}/ \partial \mathcal{P}$ is called the set of functionals over $\mathcal{P}$ and its elements are denote by $\tilde{f}=\int f dx$.

\begin{proposition}\label{prop-properties}
	The following properties are satisfied:
	\begin{enumerate}[(a)]
		\item The derivation vector fields $\partial_a$ defined by \eqref{derivation} commute with $\partial$, i.e., $[\partial_a,\partial]=0$.
		%\item The set $\left\{\partial_{a'}:a\in \mathcal{P}\right\}$ is a Lie subalgebra of  $\mathop{\rm der}\nolimits^*(\mathcal{P})$. In fact, it is satisfied that$$[\partial_{a'},\partial_{b'}]=\partial_{\partial_{a'}b'-\partial_{b'}a'}=\partial_{\left(\partial_{a'}b-\partial_{b'}a\right)'},$$
		%for all $a,b\in \mathcal{P}$.
		\item A derivation vector field $\partial_a$ acts on a functional $\F=\int f dx$, with $f\in \mathcal{P}$, according to the formula
		\begin{equation*}
			\partial_a \F=\int (\partial_a f)dx.
		\end{equation*}
		\item The following relation holds
		\begin{equation*}
			\partial_a \F=\int \sum_{i=0}^{n-2}a_i \frac{\delta f}{\delta u_i}dx,
		\end{equation*}
		where 
		$\ds\frac{\delta f}{\delta u_i}=\sum_{m=0}^\infty(-\partial)^m \left(\frac{\partial f}{\partial u_i^{(m)}}\right)$ is the standard variational derivative.
	\end{enumerate}
\end{proposition}

From an algebraic point of view, the set of functionals is given by $\bar{\mathcal{P}}=\mathcal{P}/\mathop{\rm Im}\nolimits(\partial)$, and the canonical projection $\pi:\mathcal{P}\rightarrow \bar{\mathcal{P}}$ is traditionally denoted by $\int dx$. For any $f,g\in \mathcal{P}$ we have the skew-symmetry property 
\begin{equation}\label{funtional-property}
		\int f\partial(g) dx=-\int \partial(f)g dx.
\end{equation}

%In the forthcoming sections we will work with functionals involving improper integrals for which we can apply the properties (c) and (d) of Proposition \ref{prop-properties}. We also need to introduce the  antiderivative operator $D^{-1}_{s}$ given by
%\begin{equation}\label{antiderivative}
	%\begin{split}
		%D^{-1}_{s}(g)(x)&=\frac{1}{2}\left(\int_{-\infty}^xg(s)ds-\int_x^{\infty}g(s)ds\right) \\
               %&=G(x)-\frac{1}{2}(G(+\infty)+G(-\infty)),
	%\end{split}
%\end{equation}
%being $G$ any antiderivative of $g$. Observe that $D^{-1}_{s}$ is a formal indefinite integral which { }``antisymmetrizes'' $g$ in the sense that $D^{-1}_{s}(g)(+\infty)=-D^{-1}_{s}(g)(-\infty)$. 

\subsection{Algebraic treatment of the Hamiltonian structure}
In order to define a Hamiltonian structure we need the following elements:
\begin{enumerate}[(1)]
	\item A Lie algebra $\mathcal{V}$ endowed with a Lie bracket $[\cdot,\cdot]$. The elements of $\mathcal{V}$ are called \emph{vector fields}.
	\item A linear space $\Omega^0$ with a left $\mathcal{V}$-module structure such that the elements of $\mathcal{V}$ act on $\Omega^0$ as left linear operators
	\begin{align*}
		\mathcal{V}\times\Omega^0  &\longrightarrow  \Omega^0 \\ 
		(\xi,\F)  & \rightsquigarrow \xi(\F)\equiv\xi \F
	\end{align*}
	and the following requirement is satisfied
	\begin{equation}\label{compatibilitycondition}
		\forall \xi_1,\xi_2\in \mathcal{V},\quad \forall \F\in\Omega^0 \Rightarrow (\xi_1\xi_2-\xi_2\xi_1)\F=[\xi_1,\xi_2]\F.
	\end{equation}
	\item A subspace $\Omega^1$ of the space $\bar{\Omega}_1=\left\{\alpha: \mathcal{V} \rightarrow \Omega^0  \right\}$ of all linear $\Omega^0$-valued functionals on $\mathcal{V}$. The space $\Omega^1$ has to satisfy the following conditions:
	%\begin{align*}
		%\alpha : \mathcal{V} &\longrightarrow \Omega^0 \\ 
		 %\xi & \rightsquigarrow \alpha(\xi)
	%\end{align*}
	\begin{enumerate}[(c1)]
		\item For all $\alpha\in\Omega^1$ there is an element $\xi\in \mathcal{V}$ such that $\alpha(\xi)\neq 0$.
		\item For all $\xi\in \mathcal{V}$ there is an element $\alpha\in\Omega^1$ such that $\alpha(\xi)\neq 0$.
		\item If $\F\in\Omega^0$ then $dF\in\Omega^1$, where $dF(\xi)\equiv\left\langle dF,\xi \right\rangle$ is defined by $dF(\xi)=\xi \F$.
	\end{enumerate}
	Hence a bilinear form which maps the elements $\alpha\in\Omega^1$ and $\xi\in \mathcal{V}$ to $\left\langle \alpha,\xi \right\rangle\equiv \alpha(\xi)$ is defined.
	\item A Hamiltonian (or Poisson) operator, i.e., a skew symmetric mapping  $\pi:\Omega^1\rightarrow \mathcal{V}$ verifying:
	\begin{enumerate}[(h1)]
		\item $\pi\Omega^1\subset \mathcal{V}$ is a Lie subalgebra.
		\item The $2$-form $\omega(\pi\alpha,\pi\beta)=\left\langle \pi\alpha,\beta \right\rangle$ is closed with respect to the differential defined by the Lie formula
		\begin{equation*}
			d\omega(\pi\alpha,\pi\beta,\pi\gamma)=(\pi\alpha)\omega(\pi\beta,\pi\gamma)-\omega([\pi\alpha,\pi\beta],\pi\gamma)+c.p.
		\end{equation*}
	\end{enumerate}
\end{enumerate}
Given a Hamiltonian structure, for every element $\F\in\Omega^0$ we can associate the vector field $\xi_\F=\pi(d\F)$. A natural Poisson bracket of two elements $\F,\G\in\Omega^0$ is also obtained by
\begin{equation}\label{Poisson-bracket}
	\left\{\F,\G\right\}=\xi_\F\G=d\G(\xi_\F)=d\G(\pi d\F)=\left\langle \pi d\F,d\G \right\rangle.
\end{equation}
It can be proved that the above bracket is skew-symmetric, satisfies the Jacobi equation and verifies $\xi_{\left\{\F,\G\right\}}=[\xi_\F,\xi_\G]$.

Thus, given a Hamiltonian $\H\in \Omega^0$, one can construct Hamiltonian differential equations as follows: for any $\F\in \Omega^0$ we have $\frac{\partial \F}{\partial t}=\xi_\H \F=\left\{\H,\F\right\}$. A vector field $\xi\in \mathcal{V}$ is called bi-Hamiltonian with respect to Hamiltonian operators $\pi_0$ and $\pi_1$ if there exist $H_0,H_1\in \Omega^0$ verifying:
$$\xi=\pi_0(dH_1)=\pi_1(dH_0).$$

A linear differential operator $\mathcal{R}:\mathcal{V}\rightarrow \mathcal{V}$ is a \emph{recursion operator} for a vector field $\xi$  if it is invariant under $\xi$, i.e., $L_\xi \mathcal{R}=0$, where $L_\xi$ is the Lie derivative along $\xi$. $\mathcal{R}$ is said to be \emph{hereditary} if for an arbitrary vector field $\xi\in \mathcal{V}$ the following relation is verified
$$L_{\mathcal{R}\xi}\mathcal{R}=\mathcal{R}L_\xi \mathcal{R}.$$

\begin{example}\label{ejemplo-mKdV}
	We shall set out briefly the bi-Hamiltonian structure for mKdV equation to be used later on. We consider $\mathcal{P}$ the real algebra of polynomials in $k$ and $\mathcal{V}$ the Lie algebra of derivations $\partial_a$, with $a\in \mathcal{P}$. Let $\Omega^0$ be the $\mathcal{V}$-module of functionals, i.e.,
	$$\Omega^0=\left\{\F=\int f ds: f\in \mathcal{P}\right\}.$$
	The vector fields $\partial_a$ act on the functionals according to
	$$\partial_a\F=\int\partial_af ds=\int a \frac{\delta f}{\delta k} ds.$$
	The space of one-forms is given by $\Omega^1=\left\{\alpha_p:p\in \mathcal{P}\right\}$, where each covector field $\alpha_p:\mathcal{V}\rightarrow \Omega^0$ is defined by $\alpha_p(\partial_a)=\int ap ds$. We consider the Hamiltonian operators $\pi_0,\pi_1$ given by
	$$\pi_0(\alpha_p)=\partial_{D_s(p)},\qquad \pi_1(\alpha_p)=\partial_{\mathcal{D}(p)}, \quad\text{where}\quad  \mathcal{D}=D_s^3+k_s D_s^{-1} k D_s+k^2 D_s.$$
	It is well-known that
	$$\xi=\pi_0(dH_1)=\pi_1(dH_0),\quad\text{where}\quad H_0=\int \frac{1}{2}k^2 ds,\quad  H_1=\int \left(\frac{1}{2}(k')^2-\frac{1}{8}k^4\right)ds$$
	is a bi-Hamiltonian system. The Hamiltonian differential equation is
	$$k_t=\xi_k=k'''+\frac{3}{2}k^2k'.$$
	Hence, the operator $\mathcal{R}=\pi_1\pi_0^{-1}:\mathcal{V}\rightarrow \mathcal{V}$ given by $\mathcal{R}(\partial_p)=\partial_{\mathcal{D}D_s^{-1}(p)}$, where $\mathcal{D}D_s^{-1}=D_s^2+k^2+k_sD_s^{-1}k$,  is a hereditary recursion operator for $\xi=\partial_{k'''+\frac{3}{2}k^2k'}$. 
	The Poisson bracket associated to $\pi_1$ acting on a pair of functionals $F=\int f ds$ and $G=\int g ds$  is given by
	\begin{equation}\label{poissonbracketpi1}
		\begin{aligned}
			\left\{F,G\right\}_{\pi_1}&=dG(\pi_1 dF)=dG \left(\pi_1 \alpha_{\frac{\delta f}{\delta k}}\right)\\
			&=\int  \left[\left(\frac{\delta f}{\delta k}\right)''\left(\frac{\delta g}{\delta k}\right)'+D^{-1}_{s}\left(k\left(\frac{\delta f}{\delta k}\right)'\right)k \left(\frac{\delta g}{\delta k}\right)'\right] ds.
		\end{aligned}
	\end{equation} 
\end{example}

%\begin{proposition}
	%El corchete definido anteriormente es, en efecto, un corchete de Poisson, esto es, es anti-simétrica y verifica la identidad de Jacobi. Además se verifica la indentidad:
	%$$\xi_{\{f,g\}}=[\xi_f,\xi_g].$$
%\end{proposition}

\section{Curves variations}
In this section we introduce the space of curves which we are going to work with. Firstly, we briefly describe some basic notions about the geometry of immersed curves in a $2$-dimensional real space form. Afterwards we will recall some properties and formulas about curve variations in which many results are stated without proof.

Let $M^2(G)$ be a $2$-dimensional real space form with sectional curvature $G$ and $\gamma:I \rightarrow M^2(G)$ be an immersed arc-length  parametrized curve in $M^2(G)$. We set  $\left\{\t,\n\right\}$ \emph{the Frenet frame} of the curve, where $\t(s)=\gamma'(s)$ is the unit tangent vector and $\n$ is the unit normal to the curve. 
The \emph{Frenet equations} are given by  
\begin{equation}
	\label{EcuacionesFrenet}
	\begin{split}
		\nabla_\t\t&=k\n, \\ \nabla_\t \n&=-k\t,
	\end{split}
\end{equation}
where $k(s)=\|\nabla_{\t(s)}\t(s)\|$ is the curvature of the curve. The fundamental theorem for these curves tells us that $k$ determines completely the curve up to isometries. Even more, given a function $k(s)$ we can always construct a   curve in $M^2(G)$, parametrized by the arc-length parameter $s$, whose curvature function is precisely $k(s)$. 

%Following \cite{langer_planar_1996} we consider the space of plane curves $\mathcal{M}_A=\{\gamma\in \mathcal{M}: \gamma\text{ is asymptotic to the $x$-axis}\}$. To say $\gamma$ is asymptotic to the $x$-axis means that there exist constants $\lambda_+, \lambda_-$ such that
%$$\lim_{s\to \pm\infty}\left(se_1-\gamma(s)\right)=\lambda_{\pm}e_1,\quad e_1=(1,0).$$ 
%Within $\mathcal{M}_A$, there is a distinguished subspace of \emph{balanced asymptotically linear curves} 
%\begin{equation}\label{phase-space}
		%\Lambda=\left\{\gamma\in \mathcal{M}_A: \Sigma(\gamma)=0\right\},
%\end{equation}
%where $\Sigma(\gamma)=\lambda_++\lambda_-$.

%The concept of ``balanced asymptotically linear curves'' was introduced in $\mathbb{R}^3$ (see \cite{langer_poisson_1991}, Sect. 2) in order to obtain suitable boundary conditions which will be useful later.

For a curve $\gamma=\gamma(t):I\subset \mathbb{R} \rightarrow M^2(G)$ (not necessarily arc-length parametrized) we will denote by $v(t)=\|\gamma'(t)\|$ the speed of $\gamma$. For the sake of simplicity the letter $\gamma$ will also denote a variation  $\gamma=\gamma(t,\varepsilon):I \times(-\zeta,\zeta)\rightarrow M^2(G)$ of $\gamma=\gamma(t,0)$. Associated with such a variation is the variation vector field $\v(t)=\v(t,0)$, where $\v=\v(t,\varepsilon)=\frac{\partial \gamma}{\partial \varepsilon}(t,\varepsilon)$. We will use $\t=\t(t,\varepsilon), v=v(t,\varepsilon)$, etc. with obvious meanings. We write $\gamma(s,\varepsilon), k(s,\varepsilon), v(s,\varepsilon),$ etc., for the corresponding arc-length reparametrizations. 
%If we restric $\v(s,\varepsilon)$ to the curves $s=cte$ and $\varepsilon=\cte$, we obtain vector fields along curves. In this setting, the notation $\nabla_\t$ and $\nabla_\v$ means the convariant derivatives of the restriction with respect to $s$ and $\varepsilon$, respectively.

In order to set up some results in this section we require to compute the formula for the variation of the curvature. The following lemma is a particular case of the corresponding lemma in \cite{langer_singer_1984} (p. 3) for curves variations in a Riemannian manifold.

\begin{lemma}\label{LemaVariaciones} 
With the above notation, the following 
assertions hold: 
\begin{enumerate}[(a)]
	\item $\ds \v(v)=\left.\frac{\partial v}{\partial \varepsilon}\right|_{\varepsilon=0}=\left\langle \nabla_\t\v,\t \right\rangle v$; 
	\item $\ds \v(k)=\left.\frac{\partial k}{\partial \varepsilon}\right|_{\varepsilon=0}= \left\langle \nabla_\t^2\v,\n \right\rangle-2k\left\langle \nabla_\t\v,\t \right\rangle+G \left\langle \v,\n \right\rangle$;
	\item $[\v,\t]=-\left\langle \nabla_\t\v,\t \right\rangle \t$.
\end{enumerate}
\end{lemma}

\begin{proposition} \label{Vkvarphi}
	Let $\gamma$ be an immersed curve in $M^2(G)$, $k$ its curvature and $\mathfrak{X}(\gamma)$ the set of differentiable vector fields along $\gamma$. If $\v\in \mathfrak{X}(\gamma)$, then
	\begin{equation}\label{Vk}  
		\v(k)=\varphi_\v'-k\rho_\v+Gg_\v,
	\end{equation}
	where $\rho_\v=\left\langle \nabla_\t\v,\t  \right\rangle$ and  $\varphi_\v=\left\langle\nabla_\t\v,\n \right\rangle$ stand for the tangential and normal components of $\nabla_\t\v$, respectively, and $g_\v=\left\langle \v,\n \right\rangle$ is the normal component of $\v$.
\end{proposition}
\begin{proof}
 We compute the first derivative of $\varphi_\v$,
 $$\varphi_\v'=\left\langle\nabla^2_\t\v,\n\right\rangle+\left\langle \nabla_\t\v,-k\t \right\rangle=\left\langle\nabla^2_\t\v,\n\right\rangle-k\rho_\v.$$
 Using Lemma~\ref{LemaVariaciones}(b) the formula \eqref{Vk} holds trivially.        
\end{proof}

\begin{definition}
	  We say that $\v\in\mathfrak{X}(\gamma)$ locally preserves the arc-length parameter along $\gamma$ if $\left.\frac{\partial v}{\partial \varepsilon}\right|_{\begin{subarray}{1} \varepsilon=0\end{subarray}}=0$. 
\end{definition}
As an immediate consequence of Lemma~\ref{LemaVariaciones} we obtain the following result. 

\begin{proposition}\label{prop-locallypreserving}
	A vector field $\v\in\mathfrak{X}(\gamma)$ locally preserves the arc-length parameter if and only if $\rho_\v=0$. 
\end{proposition}

Consider 
\begin{equation}\label{espacio-curvas}
		\Lambda_G=\{\gamma:I\rightarrow M^2(G) \text{ such that } \gamma\text{ is arc-length parametrized}\}
\end{equation}
the space of arc-length parametrized curves in $M^2(G)$. It is easy to see that $T_\gamma\Lambda_G$ can be identified with the set of all vector fields associated with variations of arc-length parametrized curves in $M^2(G)$ starting from $\gamma$. It is clear that a vector field in $T_\gamma\Lambda_G$ locally preserves the arc-length parameter. The converse is also true.

\begin{proposition}\label{prop-relacion-fg}%\label{tangenteM} 
	A vector field along $\gamma\in \Lambda_G$ is tangent to $\Lambda_G$ if and only if it locally preserves the arc-length parameter, i.e.,
	\begin{equation}
		T_\gamma\Lambda_G=\left\{\v\in\mathfrak{X}(\gamma): \left\langle \nabla_\t \v,\t \right\rangle=0\right\}. 
	\end{equation}
	Therefore, $\v=f_\v\t+g_\v\n\in T_\gamma\Lambda_G$ if and only if  
	\begin{equation} \label{relacion-fg} 
		f_\v=D^{-1}_{s}(kg_\v),
	\end{equation}
	where $D^{-1}_{s}$ is a formal indefinite $s$-integral.
\end{proposition}
Therefore, a tangent vector field $\v$ is completely determined by a differentiable function $g_\v$ and a constant. Throughout the paper such a constant will be assumed to be zero. From Propositions \ref{Vkvarphi}, \ref{prop-locallypreserving} and \ref{prop-relacion-fg} it is straightforward to check that if $\v\in T_\gamma\Lambda_G$ then
\begin{equation}
  \label{VKphi}
  \varphi_\v=g_\v'+kD^{-1}_{s}(kg_\v)\quad\text{and}\quad
  \v(k)=g_\v''+k'D^{-1}_{s}(kg_\v)+(k^2+G)g_\v=\mathcal{D}D_s^{-1}(g_\v)+G g_\v, 
\end{equation}
where $\mathcal{D}$ is given in Example \ref{ejemplo-mKdV}.

Finally, we shall deduce the first variation formula for an action $S:\Lambda_G \rightarrow \mathbb{R}$ given by $S(\gamma)=\int_I L(k_{\gamma},k'_{\gamma},\ldots) ds$, where $k_{\gamma}$ is the curvature function of the curve $\gamma$ and $L$ is a differentiable function. Let $\v=\left.\frac{d}{d\varepsilon}\right|_{\varepsilon=0}\gamma_\varepsilon$ be a variation vector field on $\gamma$, where $\gamma_\varepsilon\in\Lambda_G$ and $\gamma_0=\gamma$, then by using Proposition \ref{Vkvarphi} and a standard argument involving some integrations by parts, the first-order variation formula becomes:

\begin{equation}\label{firstvariation}\renewcommand*{\arraystretch}{5}
	\begin{aligned}
		(\delta S)_\gamma(\v)&=\left.\frac{d}{d\varepsilon}\right|_{\varepsilon=0}S(\gamma_\varepsilon)=\left.\frac{d}{d\varepsilon}\right|_{\varepsilon=0}\int_I L(k_{\gamma_\varepsilon},k'_{\gamma_\varepsilon},\ldots) ds \\
		&=\int_I \left.\frac{d}{d\varepsilon}\right|_{\varepsilon=0}L(k_{\gamma_\varepsilon},k'_{\gamma_\varepsilon},\ldots) ds=\int_I \v(k_\gamma)\frac{\delta L}{\delta k_{\gamma}} ds \\
		&=\int_I (\varphi_\v'+G g_\v)\frac{\delta L}{\delta k_{\gamma}} ds,
	\end{aligned}
\end{equation}
where $\frac{\delta L}{\delta k}$ stands for the usual variational derivative.

%Equation~\eqref{Vk} will play an important role in our research because in this equation appears the Hamiltonian operator $\mathcal{D}$ of the well-known second KdV Poisson structure. 

%Despite we are mainly interested in the 
%formal algebraic and geometric structures of our systems
%%We will specify the space of null curves which we are going to work with.In order
%%to specify a suitable space of null curves we follow~\cite{0795.35115} and 
%we will restrict our attention only to those Cartan curves whose curvature 
%is in a suitable phase space $P$ for KdV equation, for instance: the Schwartz space of the rapidly decreasing functions (as in \cite{0795.35115}), smooth periodic function, etc. 

\section{A Lie algebra structure on variation vector fields}  
Let $\gamma:I \rightarrow M^2(G)$ be a curve with curvature function $k$ and let $\mathcal{P}=\mathbb{R}[k^{(m)}:m\in \mathbb{N}]$ be the real algebra of polynomials in $k$ and their derivatives of arbitrary order, where $k^{(m)}=k^{(m)}(s)$, and set
$$\mathfrak{X}_{\mathcal{P}}(\gamma)=\left\{\v=f\t+g\n\in \mathfrak{X}(\gamma):f,g\in \mathcal{P}\right\},$$
and
$$T_{\mathcal{P},\gamma}(\Lambda_G)=T_\gamma(\Lambda_G)\cap \mathfrak{X}_{\mathcal{P}}(\gamma)=\left\{\v=f\t+g\n\in \mathfrak{X}_{\mathcal{P}}(\gamma):f=D_s^{-1}(kg)\in \mathcal{P}_0\right\}.$$
Motivated by Lemma \ref{LemaVariaciones} and Proposition \ref{Vkvarphi}, given $\v\in \mathfrak{X}_{\mathcal{P}}(\gamma)$, we denote by $D_\v$ the unique derivation on $\mathfrak{X}_{\mathcal{P}}(\gamma)$ verifying:
\begin{equation}\label{formulas-derivation}
		\v(k)=\varphi_\v'-k\rho_\v+Gg_\v; \quad \v(f')=\v(f)'+\rho_\v f'; \quad D_\v \t = \varphi_\v \n; \quad  D_\v \n = -\varphi_\v \t.
\end{equation}
%\begin{equation}
	%\begin{aligned}
		%D_\v f &= \v(f);  \quad \v(k)=\varphi'_\v; \\ 
		%D_\v \t &= \varphi_\v \n; \quad  D_\v \n = -\varphi_\v \t; \\
		%D_\v (f\w) &= \v(f)\w+f D_\v\w,
	%\end{aligned}
%\end{equation}
Thus, if $\v=f_\v\t+g_\v\n$ and $\w=f_\w\t+g_\w\n$ are variation vector fields in $\mathfrak{X}_{\mathcal{P}}(\gamma)$, we obtain the following formula:
\begin{equation}
	D_\v\w=\left(\v(f_\w)-g_\w\varphi_\v\right)\t+\left(\v(g_\w)+f_\w\varphi_\v \right)\n.
\end{equation}
Furthermore, it is straightforward to check that $D_\v$ is compatible with the metric, i.e.
$$\v \left\langle \w_1,\w_2 \right\rangle=\left\langle D_\v \w_1,\w_2 \right\rangle+\left\langle \w_2,D_\v\w_2 \right\rangle,$$
where $\w_1,\w_2\in\mathfrak{X}_{\mathcal{P}}(\gamma)$. We state the main result of this section.

\begin{theorem}\label{Lie-Bracket}
	Let $\gamma$ be a curve in $\Lambda_G$ and consider $[\cdot,\cdot]_{\gamma}:\mathfrak{X}_{\mathcal{P}}(\gamma)\times \mathfrak{X}_{\mathcal{P}}(\gamma)\rightarrow \mathfrak{X}_{\mathcal{P}}(\gamma)$ the map given by
	$$[\v,\w]_{\gamma}=D_{\v}\w-D_{\w}\v.$$  Then:
	\begin{enumerate}[(a)]
		\item $[\v,\w]_{\gamma}(f)=\v\w(f)-\w\v(f)$ for all $f\in \mathcal{P}$.
		\item $[,]_{\gamma}$ is skew-symmetric.
		\item $[,]_{\gamma}$ satisfies the Jacobi identity.
		\item $[,]_{\gamma}$ is closed for elements in $T_{\mathcal{P},\gamma}(\Lambda_G)$, i.e., if $\v,\w \in T_{\mathcal{P},\gamma}(\Lambda_G)$, then $[\v,\w]_{\gamma}\in T_{\mathcal{P},\gamma}(\Lambda_G)$.
	\end{enumerate}
\end{theorem}
\begin{proof}
To streamline the notation in the proof, given two vector fields $\v_i=f_i\t+g_i\n$ and $\v_j=f_j\t+g_j\n$, set $\v_{ij}=[\v_i,\v_j]_\gamma=f_{ij}\t+g_{ij}\n$, where 
\begin{equation}\label{formulas0}
	\begin{aligned}
		f_{ij} & =\v_i(f_j)-\v_j(f_i)+g_i\varphi_j-g_j\varphi_i\\
		g_{ij} & =\v_i(g_j)-\v_j(g_i)+f_j\varphi_i-f_i\varphi_j
	\end{aligned}
\end{equation}
being $\varphi_i=\varphi_{\v_i}$ and $\varphi_j=\varphi_{\v_j}$.
Furthermore, if we set $\rho_i=\rho_{\v_i}$ and $\rho_j=\rho_{\v_j}$, then some straightforward computations led to
\begin{equation}\label{formulas1}
	\begin{aligned}
		\varphi_{D_{\v_i}\v_j}&=\v_i(\varphi_j)+\rho_i\varphi_j+\rho_j\varphi_i-Gg_if_j, \\
		\rho_{D_{\v_i}\v_j}&=\v_i(\rho_j)+\rho_i\rho_j-\varphi_i\varphi_j+Gg_ig_j,
	\end{aligned}
\end{equation}
from which we can infer that
\begin{equation}\label{formulas2}
	\varphi_{ij}=\v_i(\varphi_j)-\v_j(\varphi_i)-G(g_if_j-g_jf_i), \qquad \rho_{ij}=\v_i(\rho_j)-\v_j(\rho_i).
\end{equation}
where $\varphi_{ij}=\varphi_{\v_{ij}}$ and $\rho_{ij}=\rho_{\v_{ij}}$.
%In addition, the following relations hold
%\begin{equation}\label{formulas3}
	 %\v_i(k)=\varphi'_i-k\rho_i,\qquad  D_s\v_i=\v_i D_s+\rho_i D_s.
%\end{equation}
%where the latter equation is obtained from \eqref{Vdekprima} bearing in mind that $k$ is the generator of $\mathcal{P}$.

To prove (a) it is sufficient to show that $[\v_1,\v_2]_\gamma(k)=\v_1\v_2(k)-\v_2\v_1(k)$. By combining formulas \eqref{formulas2} and \eqref{formulas-derivation} we obtain
\begin{equation*}
	[\v_1,\v_2]_\gamma(k)=\varphi'_{12}-k\rho_{12}+G g_{12}=\v_1(\varphi'_2-k\rho_2+G g_2)-\v_2(\varphi'_1-k\rho_1+G g_1)=\v_1\v_2(k)-\v_2\v_1(k).
\end{equation*}
The item (b) is deduced directly from the skew-symmetry of the formulas given in \eqref{formulas0}. In order to obtain (c) we use  (a) together with relations $\v_i(\varphi_j)=\v_i(k)f_j+k\v_i(f_j)+\v_i(g_j')$, hence, if we set  $\v_{ijk}=[\v_i,[\v_j,\v_k]_\gamma]_\gamma=f_{ijk}\t+g_{ijk}\n$, we obtain 
\begin{equation*}
	\begin{aligned}
		f_{ijk}&=\v_i\v_j(f_k)-\v_i\v_k(f_j)+\v_i(g_j)\varphi_k+\v_i(\varphi_k)g_j-\v_i(g_k)\varphi_j-\v_i(\varphi_j)g_k-\v_j(g_k)\varphi_i+\v_k(g_j)\varphi_i \\
		&\quad+\v_j(\varphi_k)g_i-\v_k(\varphi_j)g_i-[\v_j,\v_k]_\gamma(f_i)+f_j\varphi_k\varphi_i-f_k\varphi_i\varphi_j-G(g_jf_k-g_kf_j)g_i, \\
		g_{ijk}&=\v_i\v_j(g_k)-\v_i\v_k(g_j)-\v_i(f_j)\varphi_k-\v_i(\varphi_k)f_j+\v_i(g_k)\varphi_j+\v_i(\varphi_j)f_k+\v_j(f_k)\varphi_i-\v_k(f_j)\varphi_i \\
		&\quad+\v_j(\varphi_k)f_i-\v_k(\varphi_j)f_i-[\v_j,\v_k]_\gamma(g_i)+g_j\varphi_k\varphi_i-g_k\varphi_i\varphi_j+G(g_jf_k-g_kf_j)f_i.
			%\displaystyle\sum_{\text{cyclic}}f_{ijk} & =  \left( \v_i(k)g_jf_k + k\v_i(f_k)g_j  +   \v_i(g_k')g_j   -   \v(k)f_jg_k - k\v_i(f_j)g_k - \v_i(g_j')g_k + g_i\varphi_{jk} \right)\\
			%\displaystyle\sum_{\text{cyclic}} g_{ijk} & =  \left( \v_{i} (\varphi_j)f_k - \v_{i} (\varphi_k)f_j - f_i\varphi_{jk}\right) 
	\end{aligned}
\end{equation*}		
Taking into consideration again \eqref{formulas2} and \eqref{formulas-derivation}, the above cyclic terms vanish, i.e.,
\[
\sum_\text{cyclic}f_{ijk}=\sum_\text{cyclic}g_{ijk}=0.
\] 
Finally, we obtain (d) from \eqref{formulas2}, since if $\v_1$ and $\v_2$ are variation vector fields in $T_\gamma(\Lambda_G)$  then $\rho_{12}=\v_1(\rho_2)-\v_2(\rho_1)=0$, and therefore $\v_{12}\in T_{\mathcal{P},\gamma}(\Lambda_G)$.
\end{proof}

%\label{Lie-Bracket}
	%The Lie bracket $\xi=[\v,\w]_{*}$  whose tangential and normal components are, respectively, 
	%\begin{equation}\label{componentes-corchete}
	%\begin{aligned}
		%f&=  \partial_{\varphi'_1}f_2-\partial_{\varphi'_2}f_1+f'_1f_2-f_2'f_1+g'_2g_1-g'_1g_2,\\
		%g&=  \partial_{\varphi'_1}g_2-\partial_{\varphi'_2}g_1+g'_1f_2-g'_2f_1,\\
	%\end{aligned}
	%\end{equation}
	%where $f_i$ and $g_i$ are, respectively, the tangential and normal components of $\xi_i$, and $\varphi_i=\varphi_{\xi_i}$.

From the above result it follows that $[,]_{\gamma}$ is a Lie bracket, ($\mathfrak{X}_{\mathcal{P}}(\gamma),[,]_{\gamma})$ is a Lie algebra and $(T_{\mathcal{P},\gamma}(\Lambda_G),[,]_{\gamma})$ is a Lie subalgebra of $\mathfrak{X}_{\mathcal{P}}(\gamma)$. Observe the condition of arc-length preserving requires anti-differentiation which does not, in general, preserve the local nature of $f$. Theorem \ref{Lie-Bracket}(d) shows that the subspace consisting of vector fields with local $f$ is if fact closed under bracket.

Each vector field $\v\in T_{\mathcal{P},\gamma}(\Lambda_G)$ can be considered as a derivation on $\mathcal{P}$ acting on the generator $k$ in the following way: $$\v(k)=\left.\frac{d}{d\varepsilon}\right|_{\varepsilon=0}k_{\gamma_{\varepsilon}}=\varphi'_{\v}+G g_\v.$$
Furthermore $\v$ commutes with $\t=D_s$ (since $\v$ locally preserves the arc-length parameter),  then $\v$ can be indentified with the derivation vector field $\partial_{\v(k)}$.  

\begin{proposition}\label{main-theorem}
	The map $\Phi_{\gamma}:T_{\mathcal{P},\gamma}(\Lambda_G) \rightarrow \mathop{\rm der}\nolimits^*(\mathcal{P})$ defined by $\Phi_\gamma(\v)=\partial_{\v(k)}$ is a homomorphism of Lie algebras.
\end{proposition}
\begin{proof}
	 Because of the linearity of the map $\Phi_{\gamma}$, to prove that it is a homomorphism, we must show that $\Phi_{\gamma}$ keeps the Lie bracket, i.e., $\Phi_{\gamma}([\v_1,\v_2]_{\gamma})=[\Phi_{\gamma}(\v_1),\Phi_{\gamma}(\v_2)]$. According  to the definition of $\Phi_\gamma$ it is sufficient to show that  $$\partial_{[\v_1,\v_2]_\gamma(k)}=\left[\partial_{\v_{1}(k)},\partial_{\v_{2}(k)}\right]=\partial_{\partial_{\v_1(k)}\v_2(k)-\partial_{\v_2(k)}\v_1(k)}.$$ 
	 But the last equality is inferred from Theorem \ref{Lie-Bracket}(a), since
	 $$[\v_1,\v_2]_\gamma(k)=\v_1\v_2(k)-\v_2\v_1(k)=\partial_{\v_1(k)}\v_2(k)-\partial_{\v_2(k)}\v_1(k).$$
	%The injectivity will be proved once we show that $\partial_{\v(k)}=0$ implies that $\v=0$. If $\partial_{\v(k)}=0$ then $\v(k)=\varphi'_\v+G g_\v=0$, and therefore $g_\v''+k' D^{-1}_{s}(kg_\v)+(k^2+G)g_\v=0$. To satisfy the last equality, the only possibility is that $g_\v=0$, since $g_\v\in\mathcal{P}$, so we conclude that $f_\v=D^{-1}_{s}(kg_\v)=0$ and $\v=0$. 
\end{proof}

\section{Hamiltonian structure for plane curve evolution equations}
%Consider the linear space of Schwartz functions $\k:\mathbb{R} \rightarrow \mathbb{R}$, i.e., $\k$ is a smooth function such that all its derivatives rapidly vanish as $|x|\to \infty$. Let $\mathcal{P}=\mathbb{R}[\k^{(m)}:m\in \mathbb{N}]$ be the real algebra of polynomials in $\k$ and their derivatives of arbitrary order, where $\k^{(m)}=\k^{(m)}(s)$. 
Consider $\Lambda_0$ the space of arc-length parametrized plane curves given by \eqref{espacio-curvas} when $G=0$. A map $\f:\Lambda_0 \rightarrow \mathcal{C}^\infty(I,\mathbb{R})$ is referred to as a scalar field on $\Lambda_0$ and $\f(\gamma)$ will be also denoted by $\f_\gamma$.
Let $\mathcal{A}$ be the algebra of $\mathcal{P}$-valued scalar fields on $\Lambda_0$, i.e., if $\f\in \mathcal{A}$, then $\f_\gamma\in \mathcal{P}$ for all $\gamma\in \Lambda_0$. In this sense, we will also understand the curvature scalar field $\k:\Lambda_0 \rightarrow \mathcal{C}^\infty(I,\mathbb{R})$  with its obvious meaning.

In the same way, a map $\V:\Lambda_0 \rightarrow \cup_{\gamma\in \Lambda_0} T_\gamma\Lambda_0$ is referred to as a vector field on $\Lambda_0$ and $\V(\gamma)$ will be also denoted by $\V_\gamma$. We shall denote the set of tangent vector fields on $\Lambda_0$ as $\mathfrak{X}(\Lambda_0)$, and within we consider the subset $\mathfrak{X}_{\mathcal{A}}(\Lambda_0)$ of vector fields having components in $\mathcal{A}$, namely,
\begin{equation}\label{vector-fields}
	\mathfrak{X}_{\mathcal{A}}(\Lambda_0)=\left\{\V=\f\T+\g \N\in \mathfrak{X}(\Lambda_0): \f,\g\in \mathcal{A}, \f_\gamma=D^{-1}_{s}(\k_\gamma\g_\gamma)\right\},
\end{equation}
the last condition being a consequence of Proposition \ref{prop-relacion-fg}. 
%understanding that $D^{-1}_{s}(\k\g)(\gamma)=D^{-1}_{s}(\k_\gamma \g_\gamma)$. It is clear that if $\f\in \mathcal{A}$ then $D^{-1}_{s}(\f)\in \mathcal{A}_0$, where $\mathcal{A}_0=\left\{\f\in \mathcal{A}:\f_\gamma\in \mathcal{P}_0\right\}$. 
Observe that if $\V\in\mathfrak{X}_{\mathcal{A}}(\Lambda_0)$, then $\V_\gamma$ is an arc-length locally preserving vector field along $\gamma$. We also denote by $\bar{\mathfrak{X}}_{\mathcal{A}}(\Lambda_0)$ the set of vector fields $\V$ such that $\V_\gamma\in \mathfrak{X}_{\mathcal{P}}(\gamma)$ (not necessarily arc-length locally preserving). Hence,
$$\bar{\mathfrak{X}}_{\mathcal{A}}(\Lambda_0)=\left\{\V=\f\T+\g \N: \f,\g\in \mathcal{A}\right\}$$

\begin{remark}
	In what follows, we shall operate with scalar fields and vector fields in the natural way, understanding that the result of the operation is again a scalar field or vector field. For instance, if $\V,\W$ are vector fields on $\Lambda_0$, then $\left\langle \V,\W \right\rangle$ is a scalar field, where $\left\langle \V,\W \right\rangle(\gamma)=\left\langle \V_\gamma,\W_\gamma \right\rangle$; or $\nabla_\T \V(\gamma)$ is again a vector field, where $\nabla_\T \V(\gamma)=\nabla_{\T_\gamma}\V_\gamma$ and so on.
\end{remark}

As a direct consequence of Theorem \ref{Lie-Bracket}, we shall define a Lie algebra structure on $\bar{\mathfrak{X}}_{\mathcal{A}}(\Lambda_0)$.

\begin{theorem}\label{Lie-Bracket2}
	The map  $[\cdot,\cdot]:\bar{\mathfrak{X}}_{\mathcal{A}}(\Lambda_0)\times \bar{\mathfrak{X}}_{\mathcal{A}}(\Lambda_0)\rightarrow \bar{\mathfrak{X}}_{\mathcal{A}}(\Lambda_0)$  given by
	$$[\V,\W](\gamma)=[\V(\gamma),\W(\gamma)]_\gamma$$  is a Lie bracket verifying:
	\begin{enumerate}[(a)]
		\item $[\V,\W](f)=\V\W(f)-\W\V(f)$ for all $f\in \mathcal{A}$.
		\item $[,]$ is closed for elements in $\mathfrak{X}_{\mathcal{A}}(\Lambda_0)$, i.e., if $\V,\W \in \mathfrak{X}_{\mathcal{A}}(\Lambda_0)$, then $[\V,\W]\in \mathfrak{X}_{\mathcal{A}}(\Lambda_0)$.
	\end{enumerate}
	Hence, $[,]$ is a Lie bracket, ($\bar{\mathfrak{X}}_{\mathcal{A}}(\Lambda_0),[,])$ is a Lie algebra and $(\mathfrak{X}_{\mathcal{A}}(\Lambda_0),[,])$ is a Lie subalgebra of $\bar{\mathfrak{X}}_{\mathcal{A}}(\Lambda_0)$.
\end{theorem}

We define $\mathop{\rm der}\nolimits (\mathcal{A})$ the set of derivations on $\mathcal{A}$ defined in the natural way and  $\mathop{\rm der}\nolimits^*(\mathcal{A})$ the Lie subalgebra of derivation vector fields. In this setting, the elements of $\mathop{\rm der}\nolimits^*(\mathcal{A})$ are given by $\partial_\a$, with $\a\in \mathcal{A}$, such that they are defined by $\partial_\a \f(\gamma)=\partial_{\a_\gamma}\f_\gamma$, for all $\f\in \mathcal{A}$. Each vector field $\V$ on $\Lambda_0$ can be considered as a derivation on $\mathcal{A}$ acting on the generator $\k$ in the following way: $$\V(\k)(\gamma)=\V_{\gamma}(\k_\gamma)=\varphi'_{\V_\gamma},$$
where $\varphi_{\V_\gamma}=\left\langle \nabla_{\T_\gamma}\V_{\gamma}, \N_\gamma \right\rangle$. 
Because of Proposition \ref{main-theorem}, the map $\Phi:\mathfrak{X}_{\mathcal{A}}(\Lambda_0) \rightarrow \mathop{\rm der}\nolimits^*(\mathcal{A})$ defined by $\Phi(\V)(\gamma)=\Phi_\gamma(\V_\gamma)$ is a homomorphism of Lie algebras. Futhermore, $\Phi$ is injective, since if $\partial_{\V(\k)}=0$ then $\V(\k)=0$, and therefore $\V_\gamma(\k_\gamma)=0$ for all $\gamma\in\Lambda_0$, which necessarily implies that $\V=0$. Therefore, $\mathop{\rm Im}\nolimits(\Phi)$ is a Lie subalgebra of the algebra of derivation vector fields  $\mathop{\rm der}\nolimits^*(\mathcal{A})$, so we conclude that the algebra of vector fields $\mathfrak{X}_{\mathcal{A}}(\Lambda_0)$ on $\Lambda_0$ can be regarded as a Lie subalgebra of derivation vector fields.
Hereinafter, if $\V$ is a variation vector field, we will understand that $$\Phi^{-1}_\gamma(\V_\gamma)=(\Phi^{-1}(\V))_\gamma.$$

We are in a position to set out the ingredients to construct a Hamiltonian structure on $\mathfrak{X}_{\mathcal{A}}(\Lambda_0)$.
Consider $\Omega^0$ the set of actions on $\Lambda_0$ whose Lagrangian density is an element of $\mathcal{A}$. Observe that an action $\int_IL ds$  can be thought of as a functional $\int L ds$ when we consider suitable conditions on the function $\k(s)$ in order to get that the boundary term vanishes when integrating by parts, and so the property \eqref{funtional-property} holds. This is why we shall consider the set of algebraic functionals
\begin{equation*}
	\Omega^0=\left\{S=\int  L ds,\quad L=L(\k,\k',\ldots)\in \mathcal{A}\right\}.
\end{equation*}
Here we understand that $S$ acts on $\Lambda_0$ as $S(\gamma)=\int  L(\k_{\gamma},\k'_{\gamma},\ldots)ds$. 
The vector fields of $\mathfrak{X}_{\mathcal{A}}(\Lambda_0)$ act on $\Omega^0$ in the form
\begin{equation}\label{relacion-modulo}
	(\V S)(\gamma)=\int  \V_\gamma(\k_\gamma) \frac{\delta L}{\delta \k_\gamma} ds=\partial_{\V_\gamma(k_\gamma)}S(\gamma)=\Phi_\gamma(\V_\gamma)(S(\gamma)).
\end{equation}
From Theorem \ref{Lie-Bracket2}(a), the compatibility condition given by \eqref{compatibilitycondition} is trivial.

Consider $\bar{\Omega}^1$ the linear space given by all the linear $\Omega^0$-valued functionals on the vector fields space $\mathfrak{X}_{\mathcal{A}}(\Lambda_0)$. For each element $\p$ of $\mathcal{A}$ we associate a covector field $\alpha_\p:\mathfrak{X}_{\mathcal{A}}(\Lambda_0)\rightarrow \Omega^0$ defined by $\alpha_\p(\V):=\int  \V(\k) \p ds$. Therefore, we deduce 
\begin{equation}\label{relacion-unoformas}
	\alpha_\p(\V)(\gamma)=\int  \V_\gamma(\k_\gamma)\p_\gamma ds=\alpha_{\p_\gamma}(\partial_{\V_\gamma(k_\gamma)})=(\alpha_{\p_\gamma}\circ \Phi_\gamma)(\V_\gamma),
\end{equation}
where $\alpha_{\p_\gamma}$ is defined in the Example \ref{ejemplo-mKdV}.
We set the space $\Omega^1$ as
\begin{equation*}
	\Omega^1=\left\{\alpha_\p\in\bar{\Omega}^1:\p\in \mathcal{A}\;\text{and}\;\exists \p_1\in \mathcal{A} \text{ such that } \p'_1=\k \p'\right\}. 
\end{equation*}
We shall show that if $S\in\Omega^0$, then $dS\in\Omega^1$. Indeed,
\begin{equation*}
	\begin{split}
		dS(\V)=\V S=\int  \V(\k)\frac{\delta L}{\delta \k}ds= \alpha_{\frac{\delta L}{\delta \k}}(\V).
	\end{split}
\end{equation*}
Since $\k\left(\frac{\delta L}{\delta \k}\right)'$ is a total derivative we obtain $dS=\alpha_{\frac{\delta L}{\delta \k}}\in\Omega^1$. Observe that we can identify $dS$ with $\delta S$, this being a consequence of \eqref{firstvariation}.
%this being a consequence of Theorem 7.36 in \cite{musso_hamiltonian_2010}.

\begin{theorem}\label{teor-hamiltonian-operator}
	The operator  $\bar\pi:\Omega^1 \rightarrow \mathfrak{X}_{\mathcal{A}}(\Lambda_0)$ defined by $$\bar\pi(\alpha_\p)(\gamma)=\Phi^{-1}_{\gamma}\pi_1(\alpha_{\p_\gamma}),$$ 
	is a Hamiltonian operator, where $\pi_1$ is given in the Example \ref{ejemplo-mKdV}.  
\end{theorem}
The proof of Theorem \ref{teor-hamiltonian-operator} is easily deduced by using  \eqref{relacion-modulo}, 
\eqref{relacion-unoformas}  and the fact that $\Phi$ is an homomorphism of algebras.
Explicitly, we find that
\begin{equation}
	\begin{aligned}
		\bar\pi(\alpha_\p)(\gamma)&=\Phi^{-1}_{\gamma}\pi_1(\alpha_{\p_\gamma})=\Phi^{-1}_{\gamma}(\partial_{\mathcal{D}(\p_\gamma)})=\Phi^{-1}_{\gamma}(\partial_{\mathcal{D}D_s^{-1}(\p'_\gamma)}) \\
		&=D^{-1}_{s}(\k_\gamma\p'_\gamma)\T_\gamma+\p'_\gamma\N_\gamma, 
	\end{aligned}
\end{equation}
where $\mathcal{D}$ is the operator described in the Example \ref{ejemplo-mKdV}.

\begin{remark}
	Automatically, we have defined a natural Poisson bracket given by \eqref{Poisson-bracket}, which in this case takes the form:
	$$\left\{S_1,S_2\right\}_{\bar\pi}=dS_2(\bar\pi(dS_1))=\bar\pi(\alpha_{\frac{\delta L_1}{\delta k}})S_2.$$
	Evaluating on $\gamma$ and using \eqref{relacion-modulo} we obtain
	\begin{equation*}
		\begin{aligned}
			\left\{S_1,S_2\right\}_{\bar\pi}(\gamma )&=\Phi_\gamma(\bar\pi(\alpha_{\frac{\delta L_1}{\delta k}})(\gamma))S_2(\gamma)=\pi_1(\alpha_{\frac{\delta L_1}{\delta k_\gamma}})S_2(\gamma)=\left\{S_1(\gamma),S_2(\gamma)\right\}_{\pi_1} \\
			&=\int  \left[\left(\frac{\delta L_1}{\delta \k_\gamma}\right)''\left(\frac{\delta L_2}{\delta \k_\gamma}\right)'+D^{-1}_{s}\left(\k_\gamma\left(\frac{\delta L_1}{\delta \k_\gamma}\right)'\right)\k_\gamma \left(\frac{\delta L_2}{\delta \k_\gamma}\right)'\right] ds,
		\end{aligned}
	\end{equation*}
	where the last equality is followed from \eqref{poissonbracketpi1}.
\end{remark}

Finally, we proceed to calculate explicitly the Hamiltonian differential equation $\gamma_t=\V_S(\gamma)$  for a given Hamiltonian functional $S\in \Omega^0$. 
%The phase space will be  the set of plane curves $\Lambda_0$, so it is convenient to work with the vector fields $\mathfrak{X}_{\mathcal{A}}(\Lambda_0)$ instead of derivation vector fields. 
Let $S=\int  L(\k,\k',\ldots)ds$ in $\Omega^0$, then the Hamiltonian vector field associated to $S$ is given by 
\begin{equation*}
	\V_S=\bar\pi(dS)=\bar\pi(\alpha_{\frac{\delta L}{\delta \k}})=\f\T+\g\N,
\end{equation*}
where $\f_\gamma=D^{-1}_{s}\left(\k_\gamma \left(\frac{\delta L}{\delta \k_\gamma}\right)'\right)$ and $\g_\gamma=\left(\frac{\delta L}{\delta \k_\gamma}\right)'$. Consequently, Hamiltonian equation takes the form
\begin{equation*}
	\gamma_t=D^{-1}_{s}\left(\k_\gamma \left(\frac{\delta L}{\delta \k_\gamma}\right)'\right) \T_\gamma+\left(\frac{\delta L}{\delta \k_\gamma}\right)'\N_\gamma.
\end{equation*}

\begin{example}\label{example1}
We consider $S_0=\int  \frac12 \k^2 ds$, then $\frac{\delta L}{\delta \k}=\k$ and the Hamiltonian equation is
	\begin{equation*}
		\gamma_t=\V_{S_0}(\gamma)=\frac{1}{2}\k^2_\gamma\T_\gamma+\k'_{\gamma}\N_\gamma.
	\end{equation*}
\end{example}
	This equation is known by \emph{planar filament equation}, and was studied in \cite{langer_planar_1996} from a different approach.

\begin{theorem}
	The operator $\bar{\mathcal{R}}:\mathfrak{X}_{\mathcal{A}}(\Lambda_0) \rightarrow \mathfrak{X}_{\mathcal{A}}(\Lambda_0)$ defined by
	$$(\bar{\mathcal{R}}\V)(\gamma)=(\Phi^{-1}_{\gamma} \mathcal{R} \Phi_{\gamma})(\V_\gamma)$$ 
	is a hereditary recursion operator for $\V_{S_0}$.
\end{theorem}
\begin{proof}
	The proof is a direct consequence of Proposition \ref{main-theorem}. In particular, it is the result of applying the following formulas:
	\begin{equation}
		\begin{aligned}
			(L_{\V}\bar{\mathcal{R}})(\W)(\gamma)&=\Phi^{-1}_\gamma \left((L_{\Phi_\gamma\V_\gamma}\mathcal{R})(\Phi_\gamma\W_\gamma)\right) \\
			%(L_{\bar{\mathcal{R}}\V}\bar{\mathcal{R}})(\V)(\gamma)&=\Phi^{-1}_\gamma \left((L_{\mathcal{R}\Phi_\gamma\V_\gamma}\mathcal{R})(\Phi_\gamma\V_\gamma)\right)
			(\bar{\mathcal{R}}\circ L_{\V}\bar{\mathcal{R}})(\W)(\gamma)&=\Phi^{-1}_\gamma \left((\mathcal{R}\circ L_{\Phi_\gamma\V_\gamma}\mathcal{R})(\Phi_\gamma\W_\gamma)\right)
		\end{aligned}
	\end{equation}
	where $L$ denotes the Lie derivative on the corresponding spaces. It follows that $\bar{\mathcal{R}}$ is a hereditary recursion operator for $\V_{S_0}$ if and only if $\mathcal{R}$ is a hereditary recursion operator for $\Phi_\gamma(\V_{S_0,\gamma})$. Nevertheless, the vector field
	$$\Phi_\gamma(\V_{S_0,\gamma})=\partial_{\V_{S_0,\gamma}(\k_\gamma)}=\partial_{\varphi'_{\V_{S_0,\gamma}}}=\partial_{\k'''_{\gamma}+\frac{3}{2}\k^2_\gamma\k'_\gamma}$$
	is the Hamiltonian vector field of the mKdV equation, proving this the desired result.
\end{proof}

%\begin{example}\label{example2}
	%We consider $S_1=\int  \left(\frac12(\k')^2-\frac18 \k^4\right) ds$, then $\frac{\delta L}{\delta \k}=-\left(\frac12 \k^3+\k''\right)$ and the Hamiltonian equation is
	%\begin{equation*}
		%\gamma_t=\left(-\k_{\gamma}\k_{\gamma}''+\frac12 (\k'_\gamma)^2-\frac38 \k_{\gamma}^4\right)\T_\gamma-\left(\frac32 \k_{\gamma}^2\k'_{\gamma}+\k_{\gamma}'''\right)\N_\gamma.
	%\end{equation*}
%\end{example}

We can calculate explicitly the recursion operator $\bar{\mathcal{R}}$ acting on a vector field $\V=\f\T+\g\N$ as
\begin{equation}
	\begin{aligned}
		(\bar{\mathcal{R}}\V)(\gamma)&=(\Phi^{-1}_{\gamma} \mathcal{R} \Phi_{\gamma})(\V_\gamma)=\Phi^{-1}_\gamma \mathcal{R} (\partial_{\mathcal{D}D_s^{-1}(\g_\gamma)})=\Phi^{-1}_{\gamma}(\partial_{(\mathcal{D}D_s^{-1})^2(\g_\gamma)}) \\
		&=D^{-1}_s(k_\gamma \mathcal{D}D_s^{-1}(\g_\gamma))\T+\mathcal{D}D_s^{-1}(\g_\gamma)\N.
	\end{aligned}
\end{equation}
Observe that it coincides with the recursion operator for planar filament equation given in \cite{langer_planar_1996}. The three first terms of the sequence are:
\begin{equation*}
	\begin{aligned}
		\V_0&=\frac{1}{2}\k^2\T+\k' \N, \\
		\V_1&=\bar{\mathcal{R}}(\V_0)= \left(-\k\k''+\frac{1}{2}\left(\k'\right)^2-\frac{3}{8}\k^4\right)\T+\left(-\k^{(3)}-\frac{3}{2} \k^2 \k'\right)\N, \\
		\V_2&=\bar{\mathcal{R}}(\V_1)=\left(\k^{(4)}\k+\frac{1}{2}\left(\k''\right)^2+\frac{5}{16}\k^6+\k^{(3)} \k'+\frac{5}{2} \k^3 \k''+\frac{5}{4} \k^2 \left(\k'\right)^2\right)\T \\
		&\qquad\qquad+\frac{1}{8} \left(8 \k^{(5)}+20 \left(\k'\right)^3+20 \k^{(3)}\k^2+80 \k \k' \k''+15 \k^4 \k'\right)\N.
	\end{aligned}
\end{equation*}
It is thus a well-known fact that recursion operators are integro-differential operators such that acting on a symmetry produces another symmetry. In fact, if we have a Hamiltonian equation with a hereditary recursion operator, then it generates a infinite hierarchy of independent vector fields commuting each other, i.e. 
$$L_{\V_n}\bar{\mathcal{R}}=0\quad\text{and}\quad [\V_n,\V_m]=0$$
for all $n,m\in \mathbb{N}$. In this sense we refer to this system as an integrable one.

\section{Conclusions}
We have given a Hamiltonian structure for plane curves motions in the context of Gelfand and Dickey approach to integrability. The key fact in order to give the Hamiltonian structure has been to define a Lie algebra structure on the whole set of local (polinomial in curvature and its derivatives) arc-length preserving variation vector fields which includes the commuting flows as a subalgebra. As far as we know, this is not done elsewhere. In addition, we show that this set can be identified with a Lie subalgebra of derivation vector fields. This identification acts as a bridge between the phase space of mKdV equation and the phase space of planar filament equation. In particular, by means of this connection, the recursion operator for PF equation is deduced from the recursion operator of the mKdV, thus showing the well-known fact that PF equation is an integrable system.  

This way of approaching is different from what has been presented in the literature so far. For simplicity, it has been studied for $2$-dimensional backgrounds for achieving an appropiate understanding, but everything points to this framework can be adapted to space of curves in $\mathbb{R}^n$ and other different ambient spaces, that will be the subject of future research.

\section*{Acknowledgements}
This work has been partially supported by MINECO (Ministerio de Economía y Competitividad) and FEDER (Fondo Europeo de Desarrollo Regional) Project MTM2012-34037.
%, and Fundación Séneca, Spain Project No. 04540/GERM/06. 
%This research is a result of the activity developed within the framework of the Programme in Support of Excellence Groups of the Región de Murcia, Spain, by Fundación Séneca, Regional Agency for Science and Technology (Regional Plan for Science and Technology 2007-2010).

\section*{References}
\bibliographystyle{elsart-num}
\bibliography{bibliography}

\end{document}